\documentclass[11pt,twoside, leqno]{article}

\usepackage{mathrsfs}
\usepackage{amssymb}
\usepackage{amsmath}
\usepackage{amsthm}
\usepackage{amsfonts}

\usepackage{color}

\allowdisplaybreaks

\def\rr{{\mathbb R}}
\def\rn{{{\rr}^n}}
\def\zz{{\mathbb Z}}

\def\nn{{\mathbb N}}
\def\fz{\infty}

\def\ccc{{\mathbb C}}

\def\cs{{\mathcal S}}

\def\cp{{\mathcal P}}
\def\az{\alpha}

\def\supp{{\rm{\,supp\,}}}
\def\esup{\mathop\mathrm{\,ess\,sup\,}}
\def\einf{\mathop\mathrm{\,ess\,inf\,}}

\def\ls{\lesssim}
\def\lz{\lambda}

\def\sz{\sigma}

\def\hs{\hspace{0.3cm}}

\def\r{\right}
\def\lf{\left}

\def\bint{{\ifinner\rlap{\bf\kern.30em--}
\int\else\rlap{\bf\kern.35em--}\int\fi}\ignorespaces}

\def\sbint{{\ifinner\rlap{\bf\kern.32em--}
\hspace{0.078cm}\int\else\rlap{\bf\kern.45em--}\int\fi}\ignorespaces}

\def\dsup{\displaystyle\sup}

\newtheorem{theorem}{Theorem}[section]
\newtheorem{lemma}[theorem]{Lemma}

\theoremstyle{definition}
\newtheorem{remark}[theorem]{Remark}
\newtheorem{definition}[theorem]{Definition}
\numberwithin{equation}{section}

\numberwithin{equation}{section}

\numberwithin{equation}{section}

\begin{document}

\arraycolsep=1pt

\title{\Large\bf The Fourier Transform of Anisotropic Hardy Spaces with Variable Exponents and Their Applications \footnotetext{\hspace{-0.35cm} {\it 2010 Mathematics Subject Classification}.
{Primary 42B20; Secondary 42B30, 46E30.}
\endgraf{\it Key words and phrases.} Anisotropy, Hardy space, atom, Fourier transform.
%\endgraf This work is partially supported by the National
%Natural Science Foundation of China (Grant Nos. 11861062 \& 11661075).
% \endgraf $^\ast$\,Corresponding author
}}
\author{ Wenhua Wang and Aiting Wang}
\date{  }
\maketitle

\vspace{-0.8cm}

\begin{center}
\begin{minipage}{13cm}\small
{\noindent{\bf Abstract} \
Let $A$ be an expansive dilation on $\mathbb{R}^n$,
and $p(\cdot):\mathbb{R}^n\rightarrow(0,\,\infty)$ be a variable exponent function satisfying the globally log-H\"{o}lder continuous condition. Let $\mathcal{H}^{p(\cdot)}_A({\mathbb {R}}^n)$ be the variable anisotropic Hardy space introduced by Liu \cite{lyy17x}. In this paper, the authors obtain
 that the Fourier transform of $f\in \mathcal{H}^{p(\cdot)}_A({\mathbb {R}}^n)$ coincides with a continuous function $F$ on $\mathbb{R}^n$ in the
sense of tempered distributions.  As applications, the authors further conclude a higher order convergence of the continuous function $F$ at the origin and then give a variant of the Hardy-Littlewood inequality in the setting
of anisotropic Hardy spaces with variable exponents.}
\end{minipage}
\end{center}

\section{Introduction}
\hskip\parindent
As is known to all, the real-variable theory of
Hardy space $H^p(\rn)$ plays an important role in various fields of analysis and {\bf PDE}s; see, for examples, \cite{cw71,fs72,lyy16,s93,s60,sa89,t17}.
A interesting and natural problem is the characterization of the Fourier transform
$\widehat{f}$ for $f\in H^p(\rn)$. For examples, Coifman \cite{c74} characterized all $\widehat{f}$ via entire functions of exponential type
for $n = 1$, where $f\in H^p(\rr)$ with $p\in(0,\,1]$. Later,
 Taibleson and Weiss \cite{tw80} proved that, for $p\in(0,\,1]$, the
Fourier transform of $f\in H^p(\rn)$ coincides with a continuous function $F$ in the sense of tempered
distributions, and there exists a positive constant $C$, such that, for any $x\in\rn$,
\begin{align}\label{e0.1}
|F(x)|\leq C \|f\|_{H^{p}(\rn)}
|x|^{n(1/{p}-1)}.
\end{align}
In 2003, Bownik \cite{b03} introduced the anisotropic Hardy space $H_A^p(\rn)$ with $p\in(0,\,\infty)$ and $A$ being a
 general expansive matrix on $\rn$.
In 2013, Bownik and Wang \cite{bw13} further extended inequalities \eqref{e0.1} to the setting of Hardy
space $H^p
_A(\rn)$ with $p\in(0,\,1]$. In 2021, Huang et al. \cite{hcy21} established the inequalities \eqref{e0.1} to the anisotropic
mixed-norm Hardy space $H^{\vec{p}}
_{\vec{a}}(\rn)$, where $\vec{p}\in(0, 1]^n$ and $\vec{a}\in[1,\,\infty)^n$.

On the other hand, variable exponent function spaces have their applications in fluid dynamics \cite{am02}, image processing \cite{clr06}, {\bf PDE}s and variational calculus \cite{cw14,dhh11,f07,t19,t17}.
Let $p(\cdot):\rn\rightarrow(0,\,\infty)$ be a variable exponent function satisfying the globally log-H\"{o}lder continuous condition (see Section \ref{s2} below for its definition). Recently, Cruz-Uribe et al. \cite{cw14} and Nakai et al. \cite{ns12} independently introduced
the variable Hardy space $H^{p(\cdot)}({\mathbb {R}}^n)$ and obtained their real-variable theory. Then Sawano \cite{s13}, Yang et al. \cite{yz16} and Zhuo et al. \cite{zy16} further contributed to the theory.
Very recently, Liu et al. \cite{lyy17x}
introduced the variable anisotropic Hardy space $\mathcal{H}^{p(\cdot)}_A({\mathbb {R}}^n)$ associated with a general expansive matrix $A$, and established its some real-variable characterizations of $\mathcal{H}_A^{p(\cdot)}(\rn)$, respectively, in terms of the atomic, the maximal functions and  the Littlewood-Paley functions characterization.

 Inspired by the previous works,
it is a natural and interesting problem to ask whether
 there is an extension to the variable exponents setting
of the inequalities \eqref{e0.1}.
In this paper we shall answer this problem affirmatively. We obtain
 the characterization of the Fourier transform on
$ \mathcal{H}^{p(\cdot)}_{A}(\rn)$. In addition, we also establish its some applications.

Precisely, this article is organized as follows.

In Section \ref{s2}, we first recall some notation and definitions
concerning expansive dilations, the variable Lebesgue space $L^{p(\cdot)}(\rn)$ and the variable anisotropic Hardy space $\mathcal{H}^{p(\cdot)}_A({\mathbb {R}}^n)$, via the non-tangential grand maximal function.

%When the exponent function $p(\cdot)$ is reduced to the constant exponent $p$, i.e., $p(\cdot):=p\in(0,\,1]$,
%the molecular decomposition of $ \mathcal{H}^{p(\cdot)}_{A}$ in Theorem \ref{t2.6} is reduced to the molecular decomposition of %$H^{p}_{A}(\rn)=H^{p,\,p}_{A}$ in \cite[Theorem 3.9]{lyy16} and also the molecular decomposition of
%$H^\vz_A({\mathbb {R}}^n)$ with $\vz(x,\,t):=t^p$ in \cite[Theorem 2.10]{lffy16}; see Remarks \ref{r2.7} below for more details.

In Section \ref{s3}, we obtain
 that the Fourier transform of $f\in \mathcal{H}^{p(\cdot)}_A({\mathbb {R}}^n)$ coincides with a continuous function $F$ on $\mathbb{R}^n$ in the
sense of tempered distributions.

In Section \ref{s4},
as applications, we further conclude a higher order convergence of the continuous function $F$ at the origin and then give a variant of the Hardy-Littlewood inequality in the setting
of anisotropic Hardy spaces with variable exponents.

Finally, we make some conventions on notation.
Let $\nn:=\{1,\, 2,\,\ldots\}$ and $\zz_+:=\{0\}\cup\nn$.
For any $\az:=(\az_1,\ldots,\az_n)\in\zz_+^n:=(\zz_+)^n$, let
$|\az|:=\az_1+\cdots+\az_n$ and
$$\partial^\az:=
\lf(\frac{\partial}{\partial x_1}\r)^{\az_1}\cdots
\lf(\frac{\partial}{\partial x_n}\r)^{\az_n}.$$
Throughout the whole paper, we denote by $C$ a \emph{positive
constant} which is independent of the main parameters, but it may
vary from line to line.  For any $q\in[1,\,\infty]$, we denote by $q'$ its conjugate index, namely, $1/q + 1/{q'}=1$.
For any $a\in\rr$, $\lfloor a\rfloor$ denotes the
\emph{maximal integer} not larger than $a$.
%We also use $C_{(\az,\,\bz,\,\ldots)}$ to denote a positive
%constant depending on the indicated parameters $\az,\,\bz,\,\ldots$.
The \emph{symbol} $D\ls F$ means that $D\le
CF$. If $D\ls F$ and $F\ls D$, we then write $D\sim F$.
If $E$ is a
subset of $\rn$, we denote by $\chi_E$ its \emph{characteristic
function}. If there are no special instructions, any space $\mathcal{X}(\rn)$ is denoted simply by $\mathcal{X}$. Denote by $\cs$   the \emph{space of all Schwartz functions} and $\cs'$
its \emph{dual space} (namely, the \emph{space of all tempered distributions}).
\begin{remark}
It is worth to pointing that this article was done totally independently of the paper ``Fourier
transform of variable anisotropic Hardy spaces with applications to
Hardy-Littlewood inequalities" by Jun Liu \cite{l21}.
A month after this
paper was submitted the journal, we learned that Jun Liu also proved independently  the same results, see \cite{l21} or (arXiv:2112.08320). He also showed that the Fourier transform of $f\in \mathcal{H}^{p(\cdot)}_A({\mathbb {R}}^n)$ coincides with a continuous function $F$ on $\mathbb{R}^n$ in the
sense of tempered distributions, and also obtained some applications. Moreover, this article is mainly inspired by Bownik et al. \cite{bw13} and Huang et al. \cite{hcy21}.
\end{remark}
%%%%%%%%%%%%%%%%%%%%%%%%%%%%%%%%%%%%%%%%%%%%%%%%%%%%%%%%%%%%%%%%%%

%%%%%%%%%%%%%%%%%%%%%% section 2 %%%%%%%%%%%%%%%%%%%%%%%%%%%%%%%%%%

%%%%%%%%%%%%%%%%%%%%%%%%%%%%%%%%%%%%%%%%%%%%%%%%%%%%%%%%%%%%%%%%%%%%%
\section{The definitions of variable anisotropic Hardy spaces \label{s2}}
\hskip\parindent
In this section, we recall the definition  of variable anisotropic Hardy space $ \mathcal{H}^{p(\cdot)}_{A}$, via the non-tangential grand maximal function $M_N(f)$.

Firstly, we recall the definition of {{expansive dilations}}
on $\rn$; see \cite[p.\,5]{b03}. A real $n\times n$ matrix $A$ is called an {\it
expansive dilation}, shortly a {\it dilation}, if
$\min_{\lz\in\sz(A)}|\lz|>1$, where $\sz(A)$ denotes the set of
all {\it eigenvalues} of $A$. Let $\lz_-$ and $\lz_+$ be two {\it positive numbers} such that
$$1<\lz_-<\min\{|\lz|:\ \lz\in\sz(A)\}\le\max\{|\lz|:\
\lz\in\sz(A)\}<\lz_+.$$ If $A$ is diagonalizable over
$\mathbb C$, we can even take $\lz_-:=\min\{|\lz|:\ \lz\in\sz(A)\}$ and
$\lz_+:=\max\{|\lz|:\ \lz\in\sz(A)\}$.
Otherwise, we need to choose them sufficiently close to these
equalities according to what we need in our arguments.

By \cite[Lemma 2.2]{b03}, for a fixed dilation $A$,
there exist a number $r\in(1,\,\fz)$ and a set $\Delta:=\{x\in\rn:\,|Px|<1\}$, where $P$ is some non-degenerate $n\times n$ matrix, such that $$\Delta\subset r\Delta\subset A\Delta,$$ and we can
assume that $|\Delta|=1$, where $|\Delta|$ denotes the
{\it $n$-dimensional Lebesgue measure} of the set $\Delta$. Let
$B_k:=A^k\Delta$ for $k\in \zz.$ Then $B_k$ is open, $$B_k\subset
rB_k\subset B_{k+1} \ \ \mathrm{and} \ \ |B_k|=b^k.$$ In this paper, let $b:=|\det A|$.
An ellipsoid $x+B_k$ for some $x\in\rn$ and $k\in\zz$ is called a {\it dilated ball}.
Denote by $\mathfrak{B}$ the set of all such dilated balls, that is,
\begin{eqnarray}\label{e2.1}
\mathfrak{B}:=\{x+B_k:\ x\in \rn,\,k\in\zz\}.
\end{eqnarray}
Throughout the whole paper, let $\sigma$ be the {\it smallest integer} such that $2B_0\subset A^\sigma B_0$
and, for any subset $E$ of $\rn$, let $E^\complement:=\rn\setminus E$. Then,
for all $k,\,j\in\zz$ with $k\le j$, it holds true that
\begin{eqnarray}
&&B_k+B_j\subset B_{j+\sz},\label{e2.3}\\
&&B_k+(B_{k+\sz})^\complement\subset
(B_k)^\complement,\label{e2.4}
\end{eqnarray}
where $E+F$ denotes the {\it algebraic sum} $\{x+y:\ x\in E,\,y\in F\}$
of  sets $E,\, F\subset \rn$.

\begin{definition}
 A \textit{quasi-norm}, associated with
dilation $A$, is a Borel measurable mapping
$\rho_{A}:\rr^{n}\to [0,\infty)$, for simplicity, denoted by
$\rho$, satisfying
\begin{enumerate}
\item[\rm{(i)}] $\rho(x)>0$ for all $x \in \rn\setminus\{ \vec 0_n\}$,
here and hereafter, $\vec 0_n$ denotes the origin of $\rn$;

\item[\rm{(ii)}] $\rho(Ax)= b\rho(x)$ for all $x\in \rr^{n}$, where, as above, $b:=|\det A|$;

\item[\rm{(iii)}] $ \rho(x+y)\le C_A\lf[\rho(x)+\rho(y)\r]$ for
all $x,\, y\in \rr^{n}$, where $C_A\in[1,\,\fz)$ is a constant independent of $x$ and $y$.
\end{enumerate}
\end{definition}

In the standard dyadic case $A:=2{{\rm I}_{n\times n}}$, $\rho(x):=|x|^n$ for
all $x\in\rn$ is
an example of homogeneous quasi-norms associated with $A$, here and hereafter, ${\rm I}_{n\times n}$ denotes the $n\times n$ {\it unit matrix},
$|\cdot|$ always denotes the {\it Euclidean norm} in $\rn$.

By \cite[Lemma 2.4]{b03}, we know that all homogeneous quasi-norms associated with a fixed dilation
$A$ are equivalent. Therefore, for a
fixed dilation $A$, in what follows, for simplicity, we
always use the {\it{step homogeneous quasi-norm}} $\rho_A$ defined by setting,  for all $x\in\rn$,
\begin{equation*}
\rho_A(x):=\sum_{k\in\zz}b^k\chi_{B_{k+1}\setminus B_k}(x)\ {\rm
if} \ x\ne \vec 0_n,\hs {\mathrm {or\ else}\hs } \rho_A(\vec 0_n):=0.
\end{equation*}
By \eqref{e2.3}, we know that, for all $x,\,y\in\rn$,
$$\rho_A(x+y)\le b^\sz[\rho_A(x)+\rho_A(y)].$$ Moreover, $(\rn,\, \rho_A,\, dx)$ is a space of
homogeneous type in the sense of Coifman and Weiss \cite{cw71,cw77},
where $dx$ denotes the $n$-dimensional Lebesgue measure. Suppose $\rho_A$ is a homogeneous quasi-norms associated with a fixed dilation
$A$. Then there exists a constant $\mathfrak{C}_A>0$ such that, for all $x\in\rn$,
\begin{align}\label{e2.4X}
\frac{1}{\mathfrak{C}_A}[\rho_A(x)]^{\lambda_-/\ln b}\leq|x|\leq \mathfrak{C}_A[\rho_A(x)]^{\lambda_+/\ln b}\ \ \ \text{if}\ \ \ \rho_A(x)\geq 1,
\end{align}
\begin{align}\label{e2.5x}
\frac{1}{\mathfrak{C}_A}[\rho_A(x)]^{\lambda_+/\ln b}\leq|x|\leq \mathfrak{C}_A[\rho_A(x)]^{\lambda_-/\ln b}\ \ \ \text{if}\ \ \ \rho_A(x)<1,
\end{align}
where $\mathfrak{C}_A$ depends only on the dilation $A$.

Now we recall that a measurable function $p(\cdot): \rn\rightarrow(0,\,\infty)$ is called a {\it variable exponent}. For any variable exponent $p(\cdot)$, let
\begin{eqnarray}\label{e2.5}
&&p_- :=\einf_{x\in\rn} p(x)\quad \mathrm{and} \ \ p_+ :=\esup_{x\in\rn} p(x).
\end{eqnarray}
Denote by $\cp$ the set of all variable exponents $p(\cdot)$ satisfying $0<p_-\leq p_+<\infty$.

Let $f$ be a measurable function on $\rn$ and $p(\cdot)\in\cp$. Then the {\it modular function} (or, for simplicity, the {\it modular}) $\varrho_{p(\cdot)}$, associated with $p(\cdot)$, is defined by setting
$$\varrho_{p(\cdot)}(f):=\int_\rn |f(x)|^{p(x)}\, dx$$
and the  {\it Luxemburg} (also called {\it Luxemburg-Nakano}) {\it quasi-norm} $\|f\|_{L^{p(\cdot)}}$ by
$$\|f\|_{L^{p(\cdot)}}:=\inf\lf\{\lambda \in(0,\,\infty):\varrho_{p(\cdot)}(f/\lambda)\leq 1\r\}.$$
Moreover, the {\it variable Lebesgue space} $L^{p(\cdot)}$ is defined to be the set of all measurable functions $f$
satisfying that $\varrho_{p(\cdot)}(f)<\infty$, equipped with the quasi-norm $\|f\|_{L^{p(\cdot)}}$.

\begin{remark}\label{r2.1}
Let $p(\cdot)\in\cp$.
\begin{enumerate}
\item[\rm{(i)}]  For any
$r\in(0,\,\infty)$ and $f\in L^{p(\cdot)}$,
$$\lf\|{|f|^r}\r\|_{L^{p(\cdot)}}=\lf\|f\r\|^r_{L^{rp(\cdot)}}.$$
Moreover, for any $\mu\in\ccc$ and $f,g\in L^{p(\cdot)}$,
$\lf\|\mu f\r\|_{L^{p(\cdot)}}=|\mu|\lf\|f\r\|_{L^{p(\cdot)}}$ and
$$\lf\|f+g\r\|^{\underline{p}}_{L^{p(\cdot)}}\leq\lf\|f\r\|^{\underline{p}}_{L^{p(\cdot)}}+
\lf\|g\r\|^{\underline{p}}_{L^{p(\cdot)}},$$
here and hereafter,
\begin{align}\label{e2.5.1}
\underline{p}:=\min\{p_-,\,1\}.
\end{align}
\item[\rm{(ii)}] From \cite{cf13}, we know that, for any function $f\in L^{p(\cdot)}$
with $\lf\|f\r\|_{L^{p(\cdot)}}>0$,
$\varrho_{p(\cdot)}(f/{\|f\|_{L^{p(\cdot)}}})=1$
and, if $\lf\|f\r\|_{L^{p(\cdot)}}\leq1$, then
$\varrho_{p(\cdot)}(f)\leq\lf\|f\r\|_{L^{p(\cdot)}}$.
\end{enumerate}
\end{remark}

A function $p(\cdot)\in \cp$ is said to satisfy the {\it globally log-H\"{o}lder continuous condition}, denoted by
$p(\cdot)\in C^{\log}$, if there exist two positive constants $C_{\log}(p)$ and $C_{\infty}$, and $p_{\infty}\in \rr$
such that, for any $x, y\in\rn$,
\begin{align*}
|p(x)-p(y)|\leq \frac{C_{\log}(p)}{\log(e+1/\rho(x-y))}
\end{align*}
and
\begin{align*}
|p(x)-p_{\infty}|\leq \frac{C_{\infty}}{\log(e+\rho(x))}.
\end{align*}

%The following variable Lorentz space $L^{p(\cdot)}$ is known as a special case of the variable Lorentz space
%$L^{p(\cdot),\,q(\cdot)}$ investigated by Kempka and Vyb\'{i}ral in \cite{kv14}.

A $C^\infty$ function $\varphi$ is said to belong to the Schwartz class $\cs$ if,
for every integer $\ell\in\zz_+$ and multi-index $\alpha$,
$\|\varphi\|_{\alpha,\ell}:=\dsup_{x\in\rn}[\rho(x)]^\ell|\partial^\az\varphi(x)|<\infty$.
The dual space of $\cs$, namely, the space of all tempered distributions on $\rn$ equipped with the weak-$\ast$
topology, is denoted by $\cs'$. For any $N\in\zz_+$, let
\begin{eqnarray*}
\cs_N:=\lf\{\varphi\in\cs:\ \|\varphi\|_{\alpha,\ell}\leq1,\ |\alpha|\leq N,\ \ \ell\leq N\r\}.
\end{eqnarray*}
In what follows, for $\varphi\in \cs$, $k\in\zz$ and $x\in\rn$, let $\varphi_k(x):= b^{-k}\varphi\lf(A^{-k}x\r)$.

\begin{definition}
Let $\varphi\in \cs$ and $f\in \cs'$. The{\it{ non-tangential maximal function}} $M_{\varphi}(f)$ with respect to $\varphi$ is defined by setting,
for any $x\in\rn$,
\begin{eqnarray*}
M_{\varphi}(f)(x):=\sup_{y\in x+B_k, k\in\zz}|f*\varphi_{k}(y)|.
\end{eqnarray*}
Moreover, for any given $N\in \nn$, the{\it{ non-tangential grand maximal function}} $M_{N}(f)$ of $f\in \cs'$ is defined by setting,
for any $x\in\rn$,
\begin{eqnarray*}
M_{N}(f)(x):=\sup_{\varphi\in \cs_N}M_{\varphi}(f)(x).
\end{eqnarray*}
\end{definition}

The following variable anisotropic Hardy space $ \mathcal{H}^{p(\cdot)}_{A}$ was introduced in \cite[Definition 2.4]{lyy17x}.
\begin{definition}\label{d2.4}
Let $p(\cdot)\in C^{\log}$, $A$ be a dilation and $N\in[\lfloor({1/\underline{p}}-1)/\ln\lambda_{-}\rfloor+2,\,\infty)$, where $\underline{p}$
is as in \eqref{e2.5.1}. The{\it{ variable anisotropic Hardy space}} $ \mathcal{H}^{p(\cdot)}_{A}$ is defined as
\begin{eqnarray*}
 \mathcal{H}^{p(\cdot)}_{A}:=\lf\{f\in \cs':M_{N}(f)\in L^{p(\cdot)}\r\}
\end{eqnarray*}
and, for any $f\in  \mathcal{H}^{p(\cdot)}_{A}$, let $\|f\|_{ \mathcal{H}^{p(\cdot)}_{A}}:=\|M_{N}(f)\|_{L^{p(\cdot)}}$.
\end{definition}
\begin{remark} Let $p(\cdot)\in C^{\log}$.
\begin{enumerate}
\item[\rm{(i)}]  When
 $p(\cdot):=p$ with $p\in(0,\,\infty)$, the space $ \mathcal{H}^{p(\cdot)}_{A}$ is reduced to the anisotropic Hardy $H^{p}_{A}$ studied by Bownik \cite{b03}.
%is reduced to \cite[Definition 3.7]{lyy16}.
\item[\rm{(ii)}]
When $A:=2{\rm I}_{n\times n}$,
the space $ \mathcal{H}^{p(\cdot)}_{A}$ is reduced to the variable Hardy space $H^{p(\cdot)}$ introduced by Nakai et al. \cite{ns12} and also Cruz-Uribe et al. \cite{cw14}.
\end{enumerate}
\end{remark}

We begin with
the following notion of anisotropic $(p(\cdot),\,q,\,s)$-atoms introduced in \cite[Definition 4.1]{lyy17}.

\begin{definition}\label{d3.1}
Let $p(\cdot)\in\cp$, $q\in(1,\,\infty]$ and
$s\in[\lfloor(1/{p_-}-1) {\ln b/\ln \lambda_-}\rfloor,\,\infty)\cap\zz_+$ with $p_-$ as in \eqref{e2.5}. An {\it anisotropic $(p(\cdot),\,q,\,s)$-atom} is a measurable function $a$ on $\rn$ satisfying
\begin{enumerate}
\item[\rm{(i)}]  (support)  $\supp a:=\overline{\{x\in\rn:a(x)\neq0\}}\subset B$, where $B\in\mathfrak{B}$ and $\mathfrak{B}$ is as in \eqref{e2.1};
\item[\rm{(ii)}] (size)  $\|a\|_{L^q}\le \frac{|B|^{1/q}}{\|\chi_B\|_{L^{p(\cdot)}}}$;
\item[\rm{(iii)}] (vanishing moment) $\int_\rn a(x)x^\alpha dx=0$ for any $\alpha\in \mathbb{Z}^n_+$ with $|\alpha|\leq s$.
\end{enumerate}
\end{definition}
In what follows, for convenience, we call an anisotropic $(p(\cdot),\,q,\,s)$-atom simply by
a $(p(\cdot),\,q,\,s)$-atom. The following variable anisotropic atomic Hardy space
was introduced in \cite[Definition 4.2]{lyy17x}
\begin{definition}\label{d3.2}
Let $p(\cdot)\in C^{\log}$, $q\in(1,\,\infty]$,
$s\in[\lfloor(1/{p_-}-1) {\ln b/\ln \lambda_-}\rfloor,\,\infty)\cap\zz_+$ with $p_-$ as in \eqref{e2.5},
 and $A$ be a dilation.
 The {\it variable anisotropic atomic Hardy space}
$ \mathcal{H}^{p(\cdot),\,q,\,s}_{A,\mathrm{atom}}$
is defined to be the set of all distributions $f\in \cs'$ satisfying that there
exist $\{\lambda_j\}_{j\in\nn}\subset\ccc$ and a sequence of
$(p(\cdot),\,q,\,s)$-atoms, $\{a_j\}_{j\in\nn}$, supported, respectively,
on $\{{B^{(j)}}\}_{j\in\nn}\subset\mathfrak{B}$ such that
%and a positive constant $\widetilde{C}$ such that
%$\sum_{i\in\nn} \chi_{x^k_i+A^{j_0}B_{\ell^k_i}}(x)\leq\widetilde{C}$ for any $x\in\rn$ and $k\in\zz$, with some
%$j_0\in\zz\, \backslash\, \nn$, and
\begin{align*}
f=\sum_{j\in\nn} \lz_{j}a_j \ \ \mathrm{in\ } \ \cs'.
\end{align*}
%where $\lambda^k_i\sim 2^k\|\chi_{x^k_i+B_{\ell^k_i}}\|_{L^{p(\cdot)}}$
%5for any $k\in\zz$ and $i\in\nn$.
%where $\lambda^k_i\sim 2^k\|\chi_{x^k_i+B_{\ell^k_i}}\|_{L^{p(\cdot)}}$
%for any $k\in\zz$ and $i\in\nn$ with the equivalent positive constants independent of $k$ and $i$.
Moreover, for any $f\in  \mathcal{H}^{p(\cdot),\,q,\,s}_{A,\mathrm{atom}}$, let
$$\|f\|_{ \mathcal{H}^{p(\cdot),\,q,\,s}_{A,\mathrm{atom}}}
:=\inf \lf\|\lf\{\sum_{j\in\nn} \lf[\frac{|\lambda_j|\chi_{{B^{(j)}}}}
{\|\chi_{{B^{(j)}}}\|_{L^{p(\cdot)}}}\r]^{\underline{p}}\r\}^{1/\underline{p}}\r\|
_{L^{p(\cdot)}},$$
where the infimum is taken over all the decompositions of $f$ as above.
\end{definition}
%%%%%%%%%%%%%%%%%%%%%%%%%%%%%%%%%%%%%%%%%%%%%%%%%%%%%%%%%%%%%%%%%%%%%%

\section{Main result\label{s3}}
\hskip\parindent
In this section, we state the main result of the article as follow:

\begin{theorem}\label{T3.1}
Let $p(\cdot)\in C^{\log}$ satisfying $0<p_-\leq p_+\leq 1$ with $p_-$, $p_+$ as in \eqref{e2.5}. Then for any $f\in \mathcal{H}^{p(\cdot)}_A$, there exists a continuous function $F$ on $\rn$ such that
$\widehat{f}=F$ in $\cs'$, and there exists a positive constant $C$ such that
 for any $x\in\rn$,
\begin{align}\label{e3.1}
|F(x)|\leq C \|f\|_{\mathcal{H}^{p(\cdot)}_A}
\max\lf\{[\rho_{A^*}(x)]^{1/{p_-}-1},\,[\rho_{A^*}(x)]^{1/{p_+}-1}\r\},
\end{align}
where $A^*$  denotes the transposed matrix of $A$ and $\widehat{f}$ is the Fourier transform of $f$.
\end{theorem}
\begin{remark}
 When
 $p(\cdot):=p$ with $p\in(0,\,1]$, this result is reduced to \cite[Theorem 1]{b03}.
\end{remark}

%we need several technical lemmas as follows.
% we begin with
%the following notion of anisotropic $(p(\cdot),\,q,\,s)$-atoms introduced in \cite[Definition 4.1]{lyy17}.
%\begin{definition}\label{d3.1}
%Let $p(\cdot)\in\cp$, $q\in(1,\,\infty]$ and
%$s\in[\lfloor(1/{p_-}-1) {\ln b/\ln \lambda_-}\rfloor,\,\infty)\cap\zz_+$ with $p_-$ as in \eqref{e2.5}. An {\it anisotropic $(p(\cdot),\,q,\,s)$-atom} is a %measurable function $a$ on $\rn$ satisfying
%\begin{enumerate}
%\item[\rm{(i)}]  $\supp a\subset B$, where $B\in\mathfrak{B}$ and $\mathfrak{B}$ is as in \eqref{e2.1};
%\item[\rm{(ii)}] $\|a\|_{L^q}\le \frac{|B|^{1/q}}{\|\chi_B\|_{L^{p(\cdot)}}}$;
%\item[\rm{(iii)}] $\int_\rn a(x)x^\alpha dx=0$ for any $\alpha\in \mathbb{Z}^n_+$ with $|\alpha|\leq s$.
%\end{enumerate}
%\end{definition}

To prove Theorem \ref{T3.1}, we need some technical lemmas.
The following lemma reveals the atomic decompositions of the variable anisotropic Hardy spaces (see \cite[Theorem 4.8]{lyy17x}).

\begin{lemma}\label{l3.1}
Let $p(\cdot)\in C^{\log}$, $q\in(\max\{p_+,\,1\},\,\infty]$
with $p_+$ as in \eqref{e2.5},
$s\in[\lfloor(1/{p_-}-1) {\ln b/\ln \lambda_-}\rfloor,\,\infty)\cap\zz_+$ with $p_-$ as in \eqref{e2.5} and
$N\in\nn\cap[\lfloor(1/{\underline{p}}-1) {\ln b/\ln \lambda_-}\rfloor+2,\,\infty)$.
Then $$\mathcal{H}^{p(\cdot)}_                                                          {A}= \mathcal{H}^{p(\cdot),\,q,\,s}_{A,\mathrm{atom}}$$ with equivalent quasi-norms.
\end{lemma}

To state following two lemmas, let us recall two basic definitions. We define the dilation operator by $D_A(f)(x)=f(Ax)$, then commute with Fourier transform by following identity for all $j\in \zz$
\begin{align}\label{e3.2}
b^j\lf(D_{A^*}^j\widehat{D_A^j f} \r)(\xi)=\widehat{f}(\xi).
\end{align}
\begin{lemma}\label{l3.2}
Let $p(\cdot)\in C^{\log}$, $q\in(1,\,\infty]$, $s\in[\lfloor(1/{p_-}-1) {\ln b/\ln \lambda_-}\rfloor,\,\infty)\cap\zz_+$ and $a$ be a $(p(\cdot),\,q,\,s)$-atom supported on $x_0+B_k$ with some $x_0 \in \rn$ and $k\in\zz$. Then there exists a constant $C$ such that, for any $x\in\rn$,
\begin{align}\label{e3.3}
\lf|\widehat{D_A^k(a)}(x)\r|\leq Cb^{-k/q}\|a\|_{L^q}\min\lf\{1,\,|x|^{s+1}\r\}.
\end{align}
\end{lemma}
\begin{proof}
From  the assumption that $\supp a\subset x_0+B_k$ with $x_0 \in \rn$ and $k\in\zz$, we obtain that
$$\text{supp}\,D_A^k(a)\subset A^{-k}x_0+B_0.$$
Let $T(\xi)$ be the degree $s$ Taylor polynomial of the function $\xi\rightarrow e^{-2\pi i\langle x,\,\xi\rangle}$ center at $A^{-k}x_0$. Using the vanishing of moments of the atom $a$ and the H\"{o}lder inequality, we have
\begin{align*}
\lf|\widehat{D_A^k(a)}(x)\r|&=\lf|\int_{\rn}D_A^k(a)(\xi)e^{-2\pi i\langle x,\,\xi\rangle}d\xi\r|\\
&=\lf|\int_{A^{-k}x_0+B_0}D_A^k(a)(\xi)\lf[e^{-2\pi i\langle x,\,\xi\rangle}-T(\xi)\r]d\xi\r|\\
&\lesssim\int_{A^{-k}x_0+B_0}\lf|a(A^k\xi)\r|\lf|\xi-A^{-k}x_0\r|^{s+1}|x|^{s+1}d\xi\\
&\lesssim|x|^{s+1}b^{-k}\int_{x_0+B_k}\lf|a(\xi)\r|d\xi\\
&\lesssim|x|^{s+1}b^{-k/q}\|a\|_{L^q},
\end{align*}
where $i:=\sqrt{-1}$.
By the fact that $\text{supp}\,D_A^k(a)\subset A^{-k}x_0+B_0$, and the H\"{o}lder inequality, we conclude that
\begin{align*}
\lf|\widehat{D_A^k(a)}(x)\r|&=\lf|\int_{\rn}D_A^k(a)(\xi)e^{-2\pi i\langle x,\,\xi\rangle}d\xi\r|\\
&\lesssim b^{-k}\int_{x_0+B_k}\lf|a(\xi)\r|d\xi\\
&\lesssim b^{-k/q}\|a\|_{L^q}.
\end{align*}
This finishes the proof of Lemma \ref{l3.2}.
\end{proof}

To show Theorem \ref{T3.1}, we also need the following essential lemma.

\begin{lemma}\label{l3.2x}
Let $p(\cdot)\in C^{\log}$ satisfying $0<p_-\leq p_+\leq 1$ with $p_-$, $p_+$ as in \eqref{e2.5},  $q\in(1,\,\infty]$ and $s\in[\lfloor(1/{p_-}-1) {\ln b/\ln \lambda_-}\rfloor,\,\infty)\cap\zz_+$. Then there exists a positive constant $C$ such that, for any $(p(\cdot),\,q,\,s)$-atom $a$  supported on $x_0+B_k$ with $x_0 \in \rn$ and $k\in\zz$, and for any $x\in\rn$,
\begin{align}\label{e3.2x}
\lf|\widehat{a}(x)\r|\leq C\max\lf\{[\rho_{A^*}(x)]^{{1}/{p_-}-1},\,[\rho_{A^*}(x)]^{{1}/{p_+}-1}\r\}.
\end{align}
\end{lemma}

\begin{proof}
 Let $a$ be a $(p(\cdot),\,q,\,s)$-atom with $\supp a \subset x_0+B_k$, where $x_0 \in \rn$ and $k\in\zz$. Then it follows from \eqref{e3.2}, \eqref{e3.3} and the size condition of $a$ that, for any $x\in \rn$,
\begin{align}\label{e3.5xx}
\lf|\widehat{a}(x)\r|&=\lf|b^{k}\lf[D_{A^*}^k \widehat{D_{A}^k}(a)\r](x)\r| \\
&=b^{k}\lf| \widehat{D_{A}^k}(a)\lf(A^{*k}x\r)\r|\nonumber\\
&\lesssim b^{(1-1/q)k}\|a\|_{L^q}\min\lf\{1,\,\lf|A^{*k}x\r|^{s+1}\r\}\nonumber\\
&\lesssim\frac{b^{k}}{\|\chi_{x_0+B_k}\|_{L^{p(\cdot)}}}\min\lf\{1,\,\lf|A^{*k}x\r|^{s+1}\r\}\nonumber\\
&\lesssim \max\lf\{b^{(1-{1}/{p_+})k},\,b^{(1-{1}/{p_-})k}\r\}
\min\lf\{1,\,\lf|A^{*k}x\r|^{s+1}\r\}\nonumber.
\end{align}
If  $\rho_{A^*}(x)<b^{-k}$, from \eqref{e2.5} and $s>(1/p_--1)\ln b/\ln\lambda_--1$, we deduce that
\begin{align}\label{e3.3}
\lf|\widehat{a}(x)\r|&\lesssim\max\lf\{b^{(1-{1}/{p_+})k},\,b^{(1-{1}/{p_-})k}\r\}\min\lf\{1,\,\lf|A^{*k}x\r|^{s+1}\r\}\nonumber\\
&\lesssim\max\lf\{b^{(1-{1}/{p_+})k},\,b^{(1-{1}/{p_-})k}\r\}b^{(1/p_--1)k}
[\rho_{A^*}(x)]^{1/p_--1}\\ \nonumber
&\thicksim\max\lf\{b^{(1-{1}/{p_+})k},\,b^{(1-{1}/{p_-})k}\r\}.\nonumber
\end{align}
which implies that \eqref{e3.2x}  holds true.

On the other hand, for any $x\in\rn$, if $\rho_{A^*}(x)\geq b^{-k}$, then, by \eqref{e3.2}, we find that \eqref{e3.2x}  holds true in this case. The concrete details being omitted. This finishes
the proof of Lemma \ref{l3.2}.
\end{proof}
\begin{lemma}\label{l3.2x}
Let $p(\cdot)\in\mathcal{P}$. Then we have, for any $\{\lambda_j\}_{j\in \nn}\subset\mathbb{C}$ and $\{B^{(j)}\}_{j\in\nn} \subset\mathfrak{B}$,
$$\sum_{j\in\nn}|\lambda_j|\leq\lf\|\lf\{\sum_{j\in\nn}\lf[\frac{|\lambda_j|\chi_{B^{(j)}}}{\|\chi_{B^{(j)}}\|_{L^{p(\cdot)}}}\r]
^{\underline{p}}\r\}^{1/\underline{p}}\r\|_{L^{p(\cdot)}}$$
where $\underline{p}$ is as in \eqref{e2.5.1}.
\end{lemma}
\begin{proof}
Since $p(\cdot)\in\mathcal{P}$ and the well-known inequality that, $\|\cdot\|_{\ell^1}\leq\|\cdot\|_{\ell^p}$
 with $p\in(0,\,1]$, then
 \begin{align*}
\sum_{j\in\nn}|\lambda_j|&=\sum_{j\in\nn}|\lambda_j| \nonumber
\lf\|\frac{\chi_{B^{(j)}}}{\|\chi_{B^{(j)}}\|_{L^{p(\cdot)}}}\r\|_{{L^{p(\cdot)}}}\\ \nonumber
&\leq\lf\|\sum_{j\in\nn}\frac{|\lambda_j|\chi_{B^{(j)}}}{\|\chi_{B^{(j)}}\|_{L^{p(\cdot)}}}
\r\|_{L^{p(\cdot)}}\\ \nonumber
&\leq\lf\|\lf\{\sum_{j\in\nn}\lf[\frac{|\lambda_j|\chi_{B^{(j)}}}{\|\chi_{B^{(j)}}\|_{L^{p(\cdot)}}}\r]
^{\underline{p}}\r\}^{1/\underline{p}}\r\|_{L^{p(\cdot)}}.\nonumber
\end{align*}
This finishes the proof of Lemma \ref{l3.2x}.
\end{proof}

\begin{proof}[Proof of Theorem \ref{T3.1}]

 Let $p(\cdot)\in C^{\log}$ satisfying $0<p_-\leq p_+\leq 1$, $q\in(1,\,\infty]$ and $s\in[\lfloor(1/{p_-}-1) {\ln b/\ln \lambda_-}\rfloor,\,\infty)\cap\zz_+.$ For any $f\in \mathcal{H}^{p(\cdot)}_{A}$, from Definition \ref{d3.2} and Lemma \ref{l3.1}, we know that there exist numbers $\{\lambda_j\}_{j\in\mathbb{N}}\subset\mathbb{C}$ and a sequence of  $(p(\cdot),\,q,\,s)$-atom, $\{a_j\}_{j\in\mathbb{N}}$, supported, respectively, on $\{x_j+B_{\ell_j}\}_{j\in\nn}\subset\mathfrak{B}$ such that
\begin{align*}
f=\sum_{j\in\nn} \lz_{j}a_j \ \ \mathrm{in\ } \ \cs'
\end{align*}
and
\begin{align}\label{e3.5}
\|f\|_{ \mathcal{H}^{p(\cdot)}_{A}}
\thicksim\inf \lf\|\lf\{\sum_{j\in\nn} \lf[\frac{|\lambda_j|\chi_{{{x_j+B_{\ell_j}}}}}{\|
\chi_{{{x_j+B_{\ell_j}}}}\|_{L^{p(\cdot)}}}\r]^{\underline{p}}\r\}^{1/\underline{p}}\r\|
_{L^{p(\cdot)}}.
\end{align}
 By the continuity of the Fourier transform on $\cs'$, we have
 \begin{align}\label{e3.8xx}
\widehat{f}=\sum_{j\in\nn} \lz_{j}\widehat{a_j} \ \ \mathrm{in\ } \ \cs',
\end{align}
In addition, by the fact that, for any $j \in \nn,\, a_j\in  L^1$, and the Hausdorff
Young inequality, we know that $\widehat{a_j}\in L^{\infty}$. From this, Lemmas \ref{l3.2} and
\ref{l3.2x}, and \eqref{e3.5}, we conclude that, for any $x\in \rn$,
 \begin{align}\label{e3.8x}
\sum_{j\in\nn} |\lz_{j}\widehat{a_j}(x) |&\lesssim C\sum_{j\in\nn} |\lz_{j}|\max\lf\{[\rho_{A^*}(x)]^{{1}/{p_-}-1},\,[\rho_{A^*}(x)]^{{1}/{p_+}-1}\r\}\nonumber\\
&\lesssim
\lf\|\lf\{\sum_{j\in\nn} \lf[\frac{|\lambda_j|\chi_{{{x_j+B_{\ell_j}}}}}{\|
\chi_{{{x_j+B_{\ell_j}}}}\|_{L^{p(\cdot)}}}\r]^{\underline{p}}\r\}^{1/\underline{p}}\r\|
_{L^{p(\cdot)}}\max\lf\{[\rho_{A^*}(x)]^{{1}/{p_-}-1}
,\,[\rho_{A^*}(x)]^{{1}/{p_+}-1}\r\}\\\nonumber
&\lesssim \|f\|_{ \mathcal{H}^{p(\cdot)}_{A}}\max\lf\{[\rho_{A^*}(x)]^{{1}/{p_-}-1}
,\,[\rho_{A^*}(x)]^{{1}/{p_+}-1}\r\}<\infty. \nonumber
\end{align}
Thus, for any $x\in\rn$, the summation $\sum_{j\in\nn} \lz_{j}\widehat{a_j}(x)$ converges absolutely on $\rn$. Without loss of generality, we may let, for any $x\in\rn$,
\begin{align*}
F(x):=\sum_{j\in\nn} \lz_{j}\widehat{a_j}(x)
\end{align*}
pointwisely and hence, for any $x\in \rn$,
 \begin{align}\label{e3.10}
|F(x)|
&\lesssim \|f\|_{\mathcal{H}^{p(\cdot)}_{A}}\max\lf\{[\rho_{A^*}(x)]^{{1}/{p_-}-1}
,\,[\rho_{A^*}(x)]^{{1}/{p_+}-1}\r\}.
\end{align}
Next, we prove that $\widehat{f}=F \ \ \mathrm{in\ } \ \cs'$. By \eqref{e3.8xx}, it suffices to show that
$$F= \sum_{j\in\nn} \lz_{j}\widehat{a_j} \ \ \mathrm{in\ } \ \cs',$$
that is, for any $\varphi \in \cs$,
 \begin{align}\label{e3.10xx}
\lim_{N\rightarrow\infty} \lf\langle\sum_{j=1}^{N} \lz_{j}\widehat{a_j} ,\,\varphi\r\rangle&=\lim_{N\rightarrow\infty} \sum_{j=1}^{N}\lz_{j}\int_{\rn}\widehat{a_j}(x)\varphi(x)dx<\infty.
\end{align}
In fact, by Lemma \ref{l3.2}, we know that there exists a positive constant $C$ such that, for any $\varphi \in \cs$, $j\in\nn$ and $N>1/p_-$,
 \begin{align*}
\lf| \int_{\rn}\widehat{a_j}(x)\varphi(x)dx\r|
&\lesssim\sum_{k=1}^{\infty}\int_{B^*_{k+1}\backslash B^*_k}
\max\lf\{[\rho_{A^*}(x)]^{{1}/{p_-}-1},\,[\rho_{A^*}(x)]^{{1}/{p_+}-1}\r\}|\varphi(x)|dx
+\|\varphi\|_{L^1}\\
&\lesssim\sum_{k=1}^{\infty}b^{({1}/{p_-}-1)k}b^k\frac{1}{(1+b^{k})^N}+\|\varphi\|_{L^1}\\
&\lesssim1,
\end{align*}
which, together with Lemma \ref{l3.2x}, \eqref{e3.5} and \eqref{e3.10xx}, implies that
 \begin{align*}
\lim_{N\rightarrow\infty} \sum_{j=N+1}^{\infty}|\lz_{j}|\lf|\int_{\rn}\widehat{a_j}(x)\varphi(x)dx\r|\lesssim\lim_{N\rightarrow\infty} \sum_{j=N+1}^{\infty}|\lz_{j}|=0,
\end{align*}
and hence $\widehat{f}=F \ \ \mathrm{in\ } \ \cs'$. Furthermore, for any given compact set $K$, there exists a positive constant $C$, depending only on $K$, such that, for any $x\in K,\,\rho_{A^*}(x)\leq C$. By this and \eqref{e3.10}, we obtain that, for any $x\in K$,
 \begin{align*}
\sum_{j\in\nn} |\lz_{j}||\widehat{a_j}(x)|
\lesssim \max\lf\{C^{{1}/{p_-}-1},\,
C^{{1}/{p_+}-1}\r\}\sum_{j\in\nn} |\lz_{j}|.
\end{align*}
Therefore, the absolute convergence of $\sum_{j\in\nn} \lz_{j}\widehat{a_j}(x)$ is uniform on compact set $K$.
By this and the fact that, for any $j\in\nn$, $\widehat{a_j}$ is a continuous function, we conclude that, for any compact set $K$,  $F$ is also a continuous function on $K$, and hence on $\rn$. This completes the proof of Theorem \ref{T3.1}.
\end{proof}
\section{Some applications}\label{s4}
\hskip\parindent
In this section, as applications of Theorem \ref{T3.1}, we obtain a higher order convergence of the continuous function $F$ in Theorem \ref{T3.1} at the origin, and establish a variant of the Hardy-Littlewood inequality on the anisotropic Hardy spaces with variable exponents. The following conclusion is the first main result of this section.
\begin{theorem}\label{t2.8}
Let $p(\cdot)\in C^{\mathrm{log}}$ satisfying $0<p_-\leq p_+\leq 1$ with $p_-$, $p_+$ as in \eqref{e2.5}, and $f \in  \mathcal{H}^{p(\cdot)}_{A}$. Then there exists a continuous function $F$ on $\rn$ such that $\widehat{f}=F$ in $\mathcal{S'}$ and
$$\lim_{|x|\rightarrow 0^+}\frac{{F}(x)}{[\rho_{A^*}(x)]^{{1}/{p_-}-1}}=0.$$
\end{theorem}

\begin{proof}
Let  $f \in  \mathcal{H}^{p(\cdot)}_{A}$. By Lemma \ref{l3.1}, we know that there exist numbers $\{\lambda_j\}_{j\in\mathbb{N}}\subset\mathbb{C}$ and a sequence of  $(p(\cdot),\,q,\,s)$-atom, $\{a_j\}_{j\in\mathbb{N}}$, supported, respectively, on $\{x_j+B_{\ell_j}\}_{j\in\nn}\subset\mathfrak{B}$ such that
\begin{align*}
f=\sum_{j\in\nn} \lz_{j}a_j \ \ \mathrm{in\ } \ \cs'
\end{align*}
and
\begin{align*}
\|f\|_{ \mathcal{H}^{p(\cdot)}_{A}}
\thicksim\inf \lf\|\lf\{\sum_{j\in\nn} \lf[\frac{|\lambda_j|\chi_{{{x_j+B_{\ell_j}}}}}{\|
\chi_{{{x_j+B_{\ell_j}}}}\|_{L^{p(\cdot)}}}\r]^{\underline{p}}\r\}^{1/\underline{p}}\r\|
_{L^{p(\cdot)}},
\end{align*}
which, together with Lemma \ref{l3.2x}, implies that
\begin{align*}
\sum_{j\in\nn}|\lambda_j|<\infty.
\end{align*}
By Theorem \ref{T3.1} , we obtain that there exists $F$ such that $\widehat{f}=F$  in $\mathcal{S'}$. Therefore, we have
$$\frac{|{F}(x)|}{[\rho_{A^*}(x)]^{{1}/{p_-}-1}}
\leq\sum_{j\in\mathbb{N}}|\lambda_j|\frac{|\widehat{a_j}(x)|}{[\rho_{A^*}(x)]
^{{1}/{p_-}-1}}.$$
By \eqref{e3.5xx} and the condition that $0<p_-\leq p_+\leq 1$, it is easy to check that, for any $j\in\nn$,
$$\lim_{|x|\rightarrow 0^+}\frac{|\widehat{a_j}(x)|}{[\rho_{A^*}(x)]^{{1}/{p_-}-1}}=0.$$
Therefore, for any $\epsilon>0$, we know that there exists $\delta>0$ such that, for any $|x|<\delta$,
\begin{align*}
\frac{|\widehat{a_j}(x)|}{[\rho_{A^*}(x)]^{{1}/{p_-}-1}}<\frac{\epsilon}
{\sum_{j\in\mathbb{N}}|\lambda_j|+1}.
\end{align*}
By this, we have, for any $|x|<\delta$,
\begin{align*}
\frac{|{F}(x)|}{[\rho_{A^*}(x)]^{{1}/{p_-}-1}}<\epsilon.
\end{align*}
Thus,
\begin{align*}
\lim_{|x|\rightarrow 0^+}\frac{{F}(x)}{[\rho_{A^*}(x)]^{{1}/{p_-}-1}}=0.
\end{align*}
This finishes the proof of Theorem \ref{t2.8}.
\end{proof}

Now we state the second result of this section.
\begin{theorem}\label{t2.8x}
Let $p(\cdot)\in C^{\mathrm{log}}$ satisfying $0<p_-\leq p_+\leq 1$ with $p_-$, $p_+$ as in \eqref{e2.5}, and $f \in  \mathcal{H}^{p(\cdot)}_{A}$. Then there exists a continuous function $F$ on $\rn$ such that $\widehat{f}=F$ in $\mathcal{S'}$ and
$$\lf[\int_{\rn}|F(x)|^{p_+}\min \lf\{[\rho_{A^*}(x)]^{p_+-1-({p_+}/{p_-})},\,[\rho_{A^*}(x)]^{p_+-2}\r\}dx\r]^{1/p_+}\leq C\|f\|_{ \mathcal{H}^{p(\cdot)}_{A}}.$$
\end{theorem}

\begin{proof}
Let $p(\cdot)\in C^{\mathrm{log}}$ satisfying $0<p_-\leq p_+\leq 1$, and $f \in  \mathcal{H}^{p(\cdot)}_{A}$. By Lemma \ref{l3.1}, we obtain that there exist numbers $\{\lambda_j\}_{j\in\mathbb{N}}\subset\mathbb{C}$ and a sequence of  $(p(\cdot),\,\infty,\,s)$-atom, $\{a_j\}_{j\in\mathbb{N}}$, supported, respectively, on $\{x_j+B_{\ell_j}\}_{j\in\nn}\subset\mathfrak{B}$ such that
\begin{align*}
f=\sum_{j\in\nn} \lz_{j}a_j \ \ \mathrm{in\ } \ \cs',
\end{align*}
and
\begin{align*}
\|f\|_{ \mathcal{H}^{p(\cdot)}_{A}}
\thicksim\inf \lf\|\lf\{\sum_{j\in\nn} \lf[\frac{|\lambda_j|\chi_{{{x_j+B_{\ell_j}}}}}{\|
\chi_{{{x_j+B_{\ell_j}}}}\|_{L^{p(\cdot)}}}\r]^{\underline{p}}\r\}^{1/\underline{p}}\r\|
_{L^{p(\cdot)}}.
\end{align*}
By Theorem \ref{T3.1}, we know that  there exists continuous $F$ on $\rn$, such that
\begin{align*}
\widehat{f}=F\ \ \mathrm{in\ } \ \cs',
\end{align*}
Therefore, we have
\begin{align*}
&\lf\{\int_{\rn}|F(x)|^{p_+}\min \lf\{[\rho_{A^*}(x)]^{p_+-1-{p_+}/{p_-}},\,[\rho_{A^*}(x)]^{p_+-2}\r\}dx\r\}^{1/p_+}\\
&\leq\lf\{\sum_{j\in\nn} |\lz_{j}|^{p_+}\int_{\rn}\lf[|\widehat{a_j}(x)|\min \lf\{[\rho_{A^*}(x)]^{1-{1}/{p_+}-{1}/{p_-}},\,[\rho_{A^*}(x)]^{1-{2}/{p_+}}\r\}\r]^{p_+}dx\r\}^{1/p_+}
\end{align*}
By Lemma \ref{l3.2x}, we only need to show that, for any $(p(\cdot),\,\infty,\,s)$-atom $a$ with $\supp a \subset x_0+B_k$, $x_0\in\rn$, $k\in\zz$,
$$\lf\{\int_{\rn}|\widehat{a}(x)|^{p_+}\lf[\min \lf\{[\rho_{A^*}(x)]^{p_+-1-{1}/{p_-}},\,[\rho_{A^*}(x)]^{p_+-2}\r\}\r]
^{p_+}dx\r\}^{1/p_+}\lesssim1.$$
Now we can write
\begin{align*}
&\lf\{\int_{\rn}|\widehat{a}(x)|^{p_+}\lf[\min \lf\{[\rho_{A^*}(x)]^{p_+-1-{1}/{p_-}},\,[\rho_{A^*}(x)]^{p_+-2}\r\}\r]
^{p_+}dx\r\}^{1/p_+}\\
&=\lf\{\int_{B_{-k}^{*}}|\widehat{a}(x)|^{p_+}\lf[\min \lf\{[\rho_{A^*}(x)]^{1-{1}/{p_+}-{1}/{p_-}},\,[\rho_{A^*}(x)]^{p_+-2}
\r\}\r]^{p_+}dx\r\}^{1/p_+}\\
&\ \ \ +\lf\{\int_{(B_{-k}^{*})^{\complement}}|\widehat{a}(x)|^{p_+}\lf[\min \lf\{[\rho_{A^*}(x)]^{1-{1}/{p_+}-{1}/{p_-}},\,[\rho_{A^*}(x)]
^{1-{2}/{p_+}}\r\}\r]^{p_+}dx\r\}^{1/p_+}\\
&=\mathrm{I+II}.
\end{align*}

For the term $\mathrm{I}$, by \eqref{e3.3},  we obtain
\begin{align*}
\mathrm{I}&\lesssim b^{k[1+{\ln\lambda_-}(s+1)/{\ln b}]}\max \lf\{b^{-k/p_-},\,b^{-k/p_+}\r\}
\lf\{\int_{\rho_{A^*}(x)<b^{-k}}\r.  \\
&\ \ \times\lf. \lf[\min \lf\{[\rho_{A^*}(x)]^{1-{1}/{p_+}-{1}/{p_-}+\lf[(s+1)\ln {\lambda_-}
/ \ln b \r]},\,[\rho_{A^*}(x)]^{1-{2}/{p_+}-{(s+1)\ln\lambda_-}
/{\ln b}}\r\}\r]^{p_+}dx\r\}^{1/p_+}\\
&\lesssim b^{k[1+{(s+1)\ln\lambda_-}/{\ln b}]}\max \lf\{b^{-k/p_-},\,b^{-k/p_+}\r\}\\
&\ \ \times\min \lf\{b^{-k[1-{1}/{p_-}+{(s+1)\ln\lambda_-} / {\ln b}]},\,b^{-k[1-{2}/{p_+}+{(s+1)\ln\lambda_-} / {\ln b}]}\r\}\\
&\thicksim1.
\end{align*}

For the term $\mathrm{II}$, from the H\"{o}lder inequality, the Plancherel theorem, the fact that $0<p_-,\,p_+\leq1$, and the size condition of $a$, we deduce that
\begin{align*}
\mathrm{II}&=\lf\{\int_{(B_{-k}^{*})^{\complement}}|\widehat{a}(x)|^{p_+}\lf[\min \lf\{[\rho_{A^*}(x)]^{1-{1}/{p_+}-{1}/{p_-}},\,[\rho_{A^*}(x)]^{1-{2}/{p_+}}\r\}\r]
^{p_+}dx\r\}^{1/p_+}\\
&\lesssim\lf(\int_{(B_{-k}^{*})^{\complement}}|\widehat{a}(x)|^{2}dx\r)^{1/2}\\
&\ \ \ \times\lf\{\int_{(B_{-k}^{*})^{\complement}}\lf[\min \lf\{[\rho_{A^*}(x)]^{1-{1}/{p_+}-{1}/{p_-}},\,[\rho_{A^*}(x)]^{1-{2}/{p_+}}\r\}\r]
^{{2p_+}/({2-p_+})}dx\r\}^{{(2-p_+)}/({2p_+})}\\
&\lesssim \lf(\int_{\rn}|a(x)|^2\r)^{1/2}\min \lf\{b^{-k({1}/{2}-{1}/{p_-})},\,b^{-k({1}/{2}-{1}/{p_+})}\r\}\\
&\lesssim \frac{|x_0+B_k|^{1/2}}{\lf\|\chi_{x_0+B_k}\r\|_{L^{p(\cdot)}}}\min \lf\{b^{-k({1}/{2}-{1}/{p_-})},\,b^{-k({1}/{2}-{1}/{p_+})}\r\}\\
&\lesssim \max \lf\{b^{k({1}/{2}-{1}/{p_-})},\,b^{k({1}/{2}-{1}/{p_+})}\r\}\min \lf\{b^{-k({1}/{2}-{1}/{p_-})},\,b^{-k({1}/{2}-{1}/{p_+})}\r\}\\
&\thicksim 1.
\end{align*}\
This completes the proof of Theorem \ref{t2.8x}.
\end{proof}

\begin{remark}
It is worth pointing out that there is a difference with Liu's paper in the proof of Theorem \ref{t2.8x}. Indeed, we use the $(p(\cdot),\,\infty,\,s)$-atom characterization of the variable anisotropic Hardy space $\mathcal{H}_A^{p(\cdot)}(\rn)$, while Liu proves \cite[Theorem 4.3]{l21} by using the
$(p(\cdot),\,q,\,s)$-atom characterization of $\mathcal{H}_A^{p(\cdot)}(\rn)$.
\end{remark}

\textbf{Acknowledgements.} The authors would like to express their
deep thanks to the referees for their very careful reading and useful
comments which improved the presentation of this article.

\bigskip

\noindent
\medskip
\noindent

 Wenhua Wang

\medskip
\noindent
School of Mathematics and Statistics\\
Wuhan University\\
Wuhan 430072, Hubei, P. R. China\\
\smallskip
\noindent{E-mail }:
\texttt{wangwhmath@163.com} (Wenhua Wang)\\

Aiting Wang

\noindent
\noindent
School of Mathematics and Statistics\\
Qinghai Minzu University\\
 Xining
810000, Qinghai, P. R. China

\smallskip
\noindent{E-mail }:
\texttt{atwangmath@163.com} (Aiting Wang)\\

\bigskip \medskip
\end{document}